\documentclass[12pt]{amsart}
\usepackage{geometry}                
\usepackage{graphicx}
\usepackage{amssymb}
\usepackage{epstopdf}
\usepackage{amsmath}
\usepackage{color}
\usepackage{hyperref}
\usepackage{pdfsync}
 
\usepackage[all]{xy}
\xyoption{matrix}
\xyoption{arrow}
\usepackage{tikz}

\usetikzlibrary{arrows,decorations.pathmorphing,backgrounds,positioning,fit,petri}

\tikzset{help lines/.style={step=#1cm,very thin, color=gray},
help lines/.default=.5} 
\tikzset{thick grid/.style={step=#1cm,thick, color=gray},
thick grid/.default=1} 

\newtheorem{theorem}{Theorem}[section]
\newtheorem{lemma}[theorem]{Lemma}
\newtheorem{corollary}[theorem]{Corollary}

 \newtheorem{innercustomthm}{Theorem}
\newenvironment{customthm}[1]
  {\renewcommand\theinnercustomthm{#1}\innercustomthm}{\endinnercustomthm}
  	
\theoremstyle{definition}

\theoremstyle{remark}
\newtheorem{remark}[theorem]{Remark}
 
\numberwithin{equation}{section}
 
\DeclareMathOperator{\Hom}{Hom}%
\DeclareMathOperator{\Ext}{Ext}%
\newcommand{\commentout}[1]{}

\newcommand{\cE}{\ensuremath{{\mathcal{E}}}}

\newcommand{\cN}{\ensuremath{{\mathcal{N}}}}

\newcommand{\cS}{\ensuremath{{\mathcal{S}}}}

\title{Probability distribution for exceptional sequences of type $A_n$}

\author{Kiyoshi Igusa}
\address{Department of Mathematics, Brandeis University, Waltham, MA 02454}
\email{igusa@brandeis.edu}
\thanks{Supported by the Simons Foundation}



\begin{document}

\begin{abstract}We determine the probability distribution for relative projective objects in an exceptional sequence of type $A_n$ of any length. We show that these events (the $j$-th object in an exceptional sequence of length $k\le n$ being relatively projective) are independent of each other and from the length of the sequence. This gives a probabilistic interpretation of the product formula for the number of exceptional sequences of length $k$ and clusters or partial clusters of size $k$ since the latter numbers are proportional to the number of signed exceptional sequences of length $k$. 
\end{abstract}

\maketitle



\section*{Introduction}\label{sec2}

Let $Q$ be a quiver of type $A_n$ with any orientation. For example, it could be
\[
	1\to 2 \leftarrow 3\leftarrow 4\to 5\to 6.
\]
A representation of a quiver is called \emph{exceptional} if it is indecomposable and rigid, i.e., it has no self-extensions. In this case, since $Q$ is a Dynkin quiver, all indecomposable representations are rigid. Recall that an \emph{exceptional sequence} is a sequence of exceptional modules $(E_k,\cdots,E_1)$, which we find convenient to number from right to left, so that
\[
	\Hom(E_i,E_j)=0=\Ext(E_i,E_j)
\]
for all $i<j$. Thus, there are no homomorphisms or extensions going from right to left. The dimension vectors of the terms $E_i$ are linearly independent. So, $k\le n$. When $k=n$, the exceptional sequence is called \emph{complete}. It is well-known that the number of complete exceptional sequences of type $A_n$ with any orientation of this quiver is $(n+1)^{n-1}$ and there is a bijection between complete exceptional sequences of type $A_n$ and rooted labeled forests with $n$ vertices \cite{GY}, \cite{IS}.

For the quivers of type $A_n$ with orientation 
\[
	1\to 2\to 3\to \cdots\to n
\]
it was shown in \cite{IS} that the vertex $v_i$ in the corresponding rooted labeled forest is ascending/descending (i.e., $v_i$ is not a root, and its label $i$ is smaller/larger than the label $j$ of its parent $v_j$) if and only if $E_i$ is relatively injective/relatively projective and not both in the complete exceptional sequence $E_\ast$. The order is reversed from that of \cite{IS} since we are numbering the terms $E_k$ backwards (from right-to-left instead of the usual left-to-right as in Theorem \ref{IS theorem}.)

For example, the exceptional sequence $(S_3,P_1,S_1)$ for $A_3$ has forest
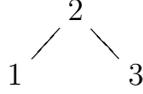
\[
\xymatrixrowsep{10pt}\xymatrixcolsep{10pt}
\xymatrix{
 & 2\ar@{-}[dl]\ar@{-}[dr]\\
1 & & 3
	}
\]

To analyze the probability distribution of relatively projective objects in an exception sequence we will go over the proof that the number of exceptional sequences of type $A_n$ and length $\ell$ is $\binom{n+1}{\ell+1}(n+1)^{\ell-1}$ while keeping track of which terms are relatively projective. We denote by $B_{\ell,k}$ the event that the term $E_k$ in an exceptional sequence $(E_\ell,\cdots,E_1)$ is relatively projective in the right perpendicular category $(E_{k-1}\oplus\cdots\oplus E_1)^\perp$.

\begin{customthm}{A}[Theorem \ref{thm: thm A}]
\label{thm A}In a random exceptional sequence $(E_\ell,\cdots,E_1)$ of length $\ell$ for a quiver of type $A_n$, 
\[
	\mathbb P(B_{\ell,k})=\frac{k+1}{n+1}.
\]
Furthermore, the events $B_{\ell,k}$, for different values of $k\le \ell$ are independent. 
\end{customthm}

For $\ell=n$, this recovers the formula (See Corollary \ref{cor: generating function for exc seq} for $\ell<n$ case.)
\[
	f_{A_n}(z)=(n+1)^{n-1}\prod_{k=1}^n \frac{(k+1)z+n-k}{n+1}=z\prod_{k=1}^{n-1} ((k+1)z+n-k)
\]
for the generating function 
\[
	f_{A_n}(z)=\sum a_pz^p
\]
where $a_p$ is the number of complete exceptional sequences with $p$ relatively projective terms. The factor 
\[
\frac{(k+1)z+n-k}{n+1}=1-(1-z)\mathbb P(B_k)
\]
is the contribution of the event $B_k$. When $k=n$ this contribution is $z$ since $\mathbb P(B_n)=1$.

By duality a similar statement (Corollary \ref{cor: injective probability}) holds for the probability distribution of the relatively injective terms where $E_i$ is relatively injective if it is an injective object in the left perpendicular category $^\perp(E_\ell\oplus\cdots\oplus E_{i+1})$ of the objects in the exceptional sequence to the left of $E_i$. This also holds for $A_n$ with any orientation. But relative projectivity and relative injectivity are strongly correlated. For
example, in the linearly oriented case, every object in a complete exceptional sequence is either relatively projective or relatively injective or both \cite{IS}.

At the end of the paper we use the results of \cite{IS} to prove the analogous statement about the probability distribution of ascending vertices in a rooted labeled forest (Corollary \ref{cor: probabilities for forests}). We also give an easy method to compute probability distributions (Theorem \ref{thm: recursive formula for fAn}) and use it to show that our main result does not hold for $D_4$.
 
The proof of Theorem \ref{thm A} comes from comparing exceptional sequence of length $k$ with subgraphs of a cyclic graph with $k+1$ out of $n+1$ edges deleted. In \cite{IS2}, a simplified version of this counting argument is obtained and used to extend Theorem \ref{thm A} to the $B_n/C_n$ case.


\section{Counting subgraphs of a cycle}\label{ss21}

We begin by counting the number of subgraphs of the $h$-cycle graph $C_h$ consisting of $h=n+1$ vertices and $h$ edges arranged in one cycle. For any $\ell\ge0$, let $L_\ell$ denote the linear graph with $\ell$ edges and $\ell+1$ vertices. 
\[
    L_3:\quad \bullet \overline{\qquad} \bullet\overline{\qquad} \bullet\overline{\qquad} \bullet
\]
We say that $L_\ell$ has \emph{length} $\ell$. For example, $L_n$ is equal to $C_h$ minus one edge.

Consider what happens when we delete $k+1$ edges from $C_h$ where $0\le k\le n$. The resulting graph $G$ will have $k+1$ linear components isomorphic to $L_{\lambda_i}$ with nonnegative lengths $\lambda_0\le \lambda_1\le\cdots\le \lambda_k$ adding up to $\sum \lambda_i=n-k$. We write $\lambda\vdash n-k$ keeping in mind that each $\lambda_i$ may be zero and we refer to $\lambda=(\lambda_0,\cdots,\lambda_k)$ as a \emph{nonnegative partition} (of $n-k$).
Given such a $\lambda$, let $\cS_h(\lambda)$ denote the set of all subgraphs $G\subset C_h$ isomorphic to $\coprod L_{\lambda_i}$. We will obtain two formulas for the size of $\cS_h(\lambda)$ by examining the set $\widetilde\cS_h(\lambda)$ of all pairs $(G,e)$ where $G\in \cS_h(\lambda)$ and $e$ is an edge of $C_h$ not in $G$.

\begin{lemma}\label{lem: size of Sh(lambda)}
For any partition $\lambda=(\lambda_0\le\cdots\le\lambda_k)$ of $n-k$ into $k+1$ possibly empty parts, the size of the set $\cS_h(\lambda)$ is
\[
	|\cS_h(\lambda)|=\frac{k!h}{\prod n_p!}
\]
where $n_p$ is the number of parts $\lambda_i$ of size $p\ge0$.
\end{lemma}

\begin{proof}
We measure the size of the set $\widetilde\cS_h(\lambda)=\{(G,j)\}$ in two ways. For each $G\in \cS_h(\lambda)$, there are $k+1$ choices for $e$. Thus $|\widetilde\cS_h(\lambda)|=(k+1)|\cS_h(\lambda)|$.

For each edge $e$ in $C_h$, the number of possible $G$ is given by a multinomial by a standard counting argument and we get
\[
	|\widetilde\cS_h(\lambda)|=(k+1)|S_h(\lambda)|=h\binom{k+1}{n_0,n_1,\cdots,n_d}
\]
where $d$ is the maximum value of $\lambda_i$. Division by $k+1$ gives the lemma.
\end{proof}

\begin{remark}\label{rem: sum of Sh(lambda)}
It follows directly from the definition that $\sum_{\lambda\vdash\,n-k}|\cS_h(\lambda)|=\binom{h}{k+1}$.
\end{remark}

A recursive formula for the size of $\cS_h(\lambda)$ is given as follows. For $(G,e)\in \widetilde\cS_h(\lambda)$, consider what happens when we add $e$ back to $G$. Then two parts of $G$ of size, say, $\lambda_i$ and $\lambda_j$ are ``fused'' together to form a new component of size $\lambda_i+\lambda_j +1$. Let $G'=G\cup \{e\}$ denote this new subgraph of $C_h$. Then $G'\in \cS_h(\lambda')$ where
\[
	\lambda'=\{
	\lambda_i+\lambda_j +1,\lambda_0,\cdots, \widehat{\lambda_i},\cdots, \widehat{\lambda_j},\cdots,\lambda_k\}.
\]
In Figure \ref{Fig: clock face}, $a,b=\lambda_i,\lambda_j$ are $1,2$. So, $c=\lambda_i+\lambda_j+1=4$. Since $a\neq b$ in this case, there are two edges in $L_c$ which can be deleted to obtain $L_a\coprod L_b$. In Figure \ref{Fig: clock face} these are $e,e'$. If $n_c'$ denotes the multiplicity of $c$ in the partition $\lambda'$, there are $2n_c'$ edges in any $G'\in \cS_h(\lambda')$ so that $G'\backslash e\in \cS(\lambda)$. If $a=b$, there is only one way to delete an edge from $L_c$ to obtain $L_a\coprod L_a$. We must delete the middle edge of $L_c$. Thus, there are $2n_c'X(a,b)$ edges in $G'$ which, when deleted, give a graph in $\cS_h(\lambda)$ where
\[
	X(a,b)=\begin{cases} \frac12 & \text{if }a= b\\
    1& \text{otherwise .}
    \end{cases}
\]

\begin{lemma}\label{lem: fusion of edges}
Let $\lambda$ be a nonnegative partition of $n-k$ into $k+1$ nonnegative parts. Then
\[
	(k+1)|\cS_h(\lambda)|=\sum_{\lambda'} |\cS_h(\lambda')| n_c' 2X(a,b)
\]
where the sum is over all $\lambda'$ obtained from $\lambda$ by fusing two parts of size, say $a,b$, in $\lambda$ into one new part of size $c=a+b+1$ in $\lambda'$.
\end{lemma}

\begin{proof}
Both sides count the size of $\widetilde\cS_h(\lambda)$. The right hand side counts the number of pairs $(G',e)$ where $e\in G'\subset C_h$ and $G=G'\backslash e\in \cS_h(\lambda)$.
\end{proof}

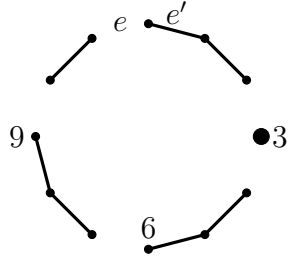
\begin{figure}[htbp]
\begin{center}
\begin{tikzpicture}[scale=1.5]
\coordinate (A0) at (0,1);
\coordinate (E) at (-.25,1);
\coordinate (E2) at (.25,1.1);
\coordinate (A1) at (0.5,0.87);
\coordinate (A2) at (0.87,.5);
\coordinate (A3) at (1,0);
\coordinate (A4) at (0.87,-.5);
\coordinate (A5) at (0.5,-0.87);
\coordinate (A6) at (0,-1);
\coordinate (A11) at (-0.5,0.87);
\coordinate (A10) at (-0.87,.5);
\coordinate (A9) at (-1,0);
\coordinate (A8) at (-0.87,-.5);
\coordinate (A7) at (-0.5,-0.87);
\draw (E) node{$e$};
\draw (E2) node{$e'$};
\begin{scope}
	\foreach \x in {A0,A1,A2,A4,A5,A6,A7,A8,A9,A10,A11}
	\draw[fill] (\x) circle[radius=1pt];
	\draw[fill] (A3) circle[radius=2pt];
	\draw (A9) node[left]{$9$};
	\draw (A3) node[right]{$3$};
	\draw (A6) node[above]{$6$};
	\draw[very thick] (A0)--(A1)--(A2);
	\draw[very thick] (A4)--(A5)--(A6);
	\draw[very thick] (A10)--(A11);
	\draw[very thick]  (A7)--(A8)--(A9);
\end{scope}
\end{tikzpicture}
\caption{Here $h=12$ with $k+1=5$ edges deleted. This leaves $n-k=7$ edges in $5$ parts with one part empty (at $3$). The parts have sizes $\lambda=(0,1,2,2,2)$. So, $n_0=1,n_1=1,n_2=3$. If the deleted edge $e$ is put back we obtain a new graph $G'$ in $\cS_{12}(\lambda')$ where $\lambda'=\{0,2,2,4\}$.}
\label{Fig: clock face}
\end{center}
\end{figure}

Lemma \ref{lem: size of Sh(lambda)} is illustrated in Figure \ref{Fig: clock face} where we have
\[
	|\cS_{12}(0,1,2,2,2)|=\frac{4!12}{3!}=48
\]
which can also be obtained visually: There are 12 possibilities for the isolated vertex, then 4 possible locations for the single edge.

To see Lemma \ref{lem: fusion of edges} in this example, note that there are four possible $\lambda'$ giving:
\begin{enumerate}
    \item $|\cS_{12}(2,2,2,2)| \,n_2' 2X(0,1)=3\cdot 4\cdot 2=24$
\item $|\cS_{12}(1,2,2,3)| \,n_3' 2X(0,2)=36\cdot 1\cdot 2=72$
    \item  $|\cS_{12}(0,2,2,4)| \,n_4' 2X(1,2)=36\cdot 1\cdot 2=72$
\item  $|\cS_{12}(0,1,2,5)|\, n_5' 2X(2,2)=72\cdot 1\cdot1=72$.
\end{enumerate}
These add up to $240=5\cdot 48$ as claimed.

The numbers $|S_h(\lambda)|$ will be used to count exceptional sequences of length $k$ for $A_n$.


\section{Counting exceptional sequences}\label{ss22}

Let $\Lambda$ be a hereditary algebra of type $A_n$ with any orientation. For $1\le k\le n$ and any nonnegative partition $\lambda=(\lambda_0\le\lambda_1\le\cdots\le\lambda_k)$ of $n-k$ into $k+1$ parts of size $\lambda_i\ge0$, let $\cN_h(\lambda)$ be the set of exceptional sequences $E_\ast=(E_k,\cdots,E_1)$ of length $k$ for $\Lambda$ whose right perpendicular category $(E_k\oplus\cdots\oplus E_1)^\perp$ is of type 
\[
	A_\lambda:=\prod_{\lambda_i>0} A_{\lambda_i}.
\]
Let $\cN_h^p(\lambda)$ be the subset of $\cN_h(\lambda)$ consisting of exceptional sequences $(E_k,\cdots,E_1)$ where $E_k$ is a projective object of $\cE'=(E_{k-1}\oplus\cdots\oplus E_1)^\perp$. Then we will prove the following.

\begin{theorem}\label{thm: Nh(lambda)} The sizes of the sets $\cN_h(\lambda)$, $\cN_h^p(\lambda)$ are
\[
	|\cN_h(\lambda)|=h^{k-1}|\cS_h(\lambda)|=\frac{k!h^k}{\prod n_p!}
\]
\[
|\cN_h^p(\lambda)|=h^{k-2}(k+1)|\cS_h(\lambda)|=\frac{(k+1)!h^{k-1}}{\prod n_p!}
\]
where $n_p$ is the number of parts $\lambda_i$ equal to $p$.
\end{theorem}

For example, take $n=k=3$. Then $\lambda=(0,0,0,0)$ with $n_0=4$. Since $k+1=h$, Theorem \ref{thm: Nh(lambda)} states that
\[
	|\cN_h^p(\lambda)|=|\cN_h(\lambda)|=|\cN_4(0,0,0,0)|=4^{2}|\cS_4(\lambda)|=\frac{3!4^3}{4!}=4^2
\]
as expected. Since $\sum_\lambda |\cS_h(\lambda)|=\binom h{k+1}$, this implies the following well-known formula:

\begin{corollary}\label{cor: number of exc seq of type A}
The number of length $k$ exceptional sequences of type $A_n$ is $\binom{h}{k+1}h^{k-1}$.
\end{corollary}

Theorem \ref{thm: Nh(lambda)} is proved by induction on $k$ starting with the case $k=1$ which illustrates the main step in the proof. If $E_1=P_i$ is the $i$th projective module, then $E_1^\perp$ consists of representations of the $A_n$ quiver whose support does not contain vertex $i$. In that case, $\lambda=(i-1,n-i)$. Anything in the $\tau$ orbit of either $P_i$ or $P_{n-i+1}$ has an equivalent perpendicular category and there are exactly $h=n+1$ elements in the union of these two $\tau$ orbits except in the case when $i=\frac h2$ in which case $P_i=P_{n-i+1}$ and there is only one $\tau$ orbit which has size $\frac h2$. Thus, for $0\le a\le b$ with $a+b+1=n$ we have $\cN_h(a,b)|=hX(a,b)$ where we recall that $X(a,b)=\frac12$ or $1$ depending on whether $a=b$ or not. Since each $\tau$ orbit has only one projective object, $\cN_h^p(a,b)=2X(a,b)$. Thus, $|\cN_h(a,b)|=|\cS_h(a,b)|$ and $|\cN_h^p(a,b)|=\frac2h|\cS_h(a,b)|$ proving the theorem in the case $k=1$. The rest of the proof is by induction on $k$ using the following lemma.

\begin{lemma}\label{lem: recursive formula for no. of k-exc seq of type An}
Given an exceptional sequence $(E_{k-1},\cdots,E_1)$ for $A_n$ of length $k-1$ with perpendicular category $\cE'=(E_{k-1}\oplus \cdots\oplus E_1)^\perp$ of type $A_{\lambda'}$, $\lambda'=(\lambda_1',\cdots,\lambda_k')$, the number of ways to add one more term ($E_k$ on the left end) to get an exceptional sequence of length $k$ with perpendicular category of type $A_\lambda$ is equal to
\[
	n_c'(c+1)X(a,b)
\]
when $\lambda=(a,b,\lambda_1',\cdots,\widehat{\lambda_p'},\cdots,\lambda_k')$, $\lambda_p'=c=a+b+1$ and $0\le a\le b$ where $X(a,b)$ is defined above and $n_c'$ is the number of parts $\lambda_p'$ which are equal to $c$. The number of such $E_k$ which are projective in the perpendicular category $\cE'$ is $2n_c'X(a,b)$.
\end{lemma}

\begin{proof} The added term $E_k$ in the exceptional sequence belongs to some $A_c$ and, on that subquiver, $E_k^\perp=A_a\times A_b$ where $a+b+1=c$ and $a,b\ge0$. There are $n_c'$ components of the perpendicular category $\cE'=\prod A_{\lambda_i'}$ of type $A_c$. Each such component has $(c+1)X(a,b)$ objects $E$ so that $E^\perp=A_a\times A_b$. Of these, $2X(a,b)$ are projective in $A_c$. Since these are all the indecomposable objects and projective indecomposable objects of the perpendicular category $\cE'$ up to isomorphism, the two formulas hold.
\end{proof}

\begin{proof}[Proof of Theorem \ref{thm: Nh(lambda)}]
The theorem states that $|\cN_h(\lambda)|=|\cS_h(\lambda)|h^{k-1}$ and $|\cN_h^p(\lambda)|=\frac{k+1}h|\cN_h(\lambda)|$. We have verified both equations for $k=1$. So, suppose that $k\ge2$. By the lemma, we have:
\[
	|\cN_h(\lambda_0,\cdots,\lambda_k)|=\sum_{\lambda'}|\cN_h(\lambda')|n_c'(c+1)X(a,b)
\]
\[
	|\cN_h^p(\lambda_0,\cdots,\lambda_k)|=\sum_{\lambda'}|\cN_h(\lambda')|n_c'2X(a,b)
\]
where both sums are over all $\lambda'$ obtained from $\lambda$ by ``fusing'' two components $a=\lambda_i$ and $b=\lambda_j$ together to form $c=a+b+1$. By induction on $k$, we have $|\cN_h(\lambda')|=h^{k-1}|\cS_h(\lambda')|$. If we insert this into the second equation we get:
\[
    |\cN_h^p(\lambda)|=h^{k-2}\sum_{\lambda'} |\cS_h(\lambda')|n_c'2X(a,b)=h^{k-2}(k+1)|\cS_h(\lambda)|=\frac{(k+1)!h^{k-1}}{\prod n_p!}
\]
where we applied Lemma \ref{lem: fusion of edges} to obtain the second equality.

To compute $|\cN_h(\lambda)|$ we need to change the index of summation in the above sum to $0\le i<j\le n$. However, different values of $i,j$ may give the same $\lambda'$. So, we need to divide by a redundancy factor $M(a,b)$:
\[
	|\cN_h(\lambda)|=\sum_{0\le i<j\le k} |\cN_h(\lambda_i+\lambda_j+1,\lambda_0,\cdots,\widehat{\lambda_i},\cdots,\widehat{\lambda_j},\cdots, \lambda_k)|\left(\frac{n_c'(\lambda_i+\lambda_j+2)X(\lambda_i,\lambda_j)}{M(\lambda_i,\lambda_j)}\right)
\]
where $c=\lambda_i+\lambda_j+1$ and $M(\lambda_i,\lambda_j)$ is the number of times this same term appear in the sum over all $i<j$. Thus
\[
	M(\lambda_i,\lambda_j)=\begin{cases} \frac12 n_{\lambda_i}(n_{\lambda_i}-1) & \text{if } {\lambda_i}={\lambda_j}\\
  n_{\lambda_i}n_{\lambda_j}  & \text{otherwise}
    \end{cases}
\]
Inserting the value $|\cN_h(\lambda')|=h^{k-1}\frac{(k-1)!}{\prod n_p'! }$  which is given by induction, we obtain:
\[
	|\cN_h(\lambda)|=\sum_{0\le i<j\le k} h^{k-1}\frac{(k-1)!}{\prod n_p'! }\left(\frac{n_c'(\lambda_i+\lambda_j+2)X(\lambda_i,\lambda_j)}{M(\lambda_i,\lambda_j)}\right).
\]
But, $\prod n_p'!$ is related to $\prod n_p!$ in the following way since $n_a'=n_a-1$ and $n_b'=n_b-1$ for $a,b=\lambda_i,\lambda_j$ when $\lambda_i\neq\lambda_j$ (or $n_a'=n_a-2$ when $\lambda_i=\lambda_j$) and $n_c'=n_c+1$:
\[
	\prod n_p'=\begin{cases} \frac{n_c'\prod n_p!}{n_{\lambda_i}n_{\lambda_j}} & \text{if } \lambda_i\neq \lambda_j\\
   \frac{n_c'\prod n_p!}{n_{\lambda_i}(n_{\lambda_i}-1)} & \text{if }\lambda_i=\lambda_j
    \end{cases}\quad= \frac{n_c' \prod n_p! X(\lambda_i,\lambda_j)}{M(\lambda_i,\lambda_j)}
\]
So,
\begin{equation}\label{eq in proof of Thm Nh(lambda)}
	|\cN_h(\lambda)|=h^{k-1}\frac{(k-1)!}{\prod n_p!}\sum_{1\le i<j\le k} (\lambda_i+\lambda_j+2)
\end{equation}
It remains to compute the sum $\sum_{i<j} (\lambda_i+\lambda_j+2)$ in \eqref{eq in proof of Thm Nh(lambda)}. To do this we first observe that, since $\lambda$ is a partition of $n-k$, 
\[
	\sum_{i=0}^k(\lambda_i+1)=n-k+k+1=n+1=h.
\]
So,
\[
	\sum_{0\le i<j\le k} (\lambda_i+\lambda_j+2)= 
	\sum_{i\neq j} \frac{\lambda_i+\lambda_j+2}2=
	\sum_{i,j=0}^k\frac{\lambda_i+\lambda_j+2}2-\sum_{i=j}\frac{\lambda_i+\lambda_j+2}2
\]
\[
	=(k+1)\sum \frac{\lambda_i+1}{2}+(k+1)\sum {\lambda_j+\frac12}-\sum(\lambda_i+1)=kh.
\]
Inserting this into \eqref{eq in proof of Thm Nh(lambda)} gives $|\cN_h(\lambda)|=h^k k!/\prod n_p!=|\cS_h(\lambda)|h^{k-1}$ as claimed.
\end{proof}


\section{Probability distribution of relatively projective terms}\label{ss23}

The following corollary and its proof is a generalization of the $k=1$ case from \cite{Ringel}. For $k\le n$, let $B_{k}$ be the event that $E_k$ is relatively projective in a complete exceptional sequence $(E_n,E_{n-1},\cdots,E_1)$. 

We recall a standard tool in probability where a \emph{partition} is a decomposition of a sample space into a countable disjoint union of measurable subsets (events).

\begin{lemma}\label{lem: partition probability}
Given a partition $X_\alpha$ and any event $Y$,
\[
	\mathbb P(Y)=\sum_\alpha \mathbb P(X_\alpha Y)=\sum_\alpha \mathbb P(X_\alpha)\mathbb P(Y|X_\alpha).
\]
\end{lemma}

\begin{lemma} For any $k\le n$ and $\lambda=(\lambda_0,\cdots,\lambda_k)$ a partition of $n-k$ into $k+1$ parts we have:
\begin{equation}\label{eq: independence of Bk,Tk}
\mathbb P\left(B_kT_k(\lambda)\right)=\frac{k+1}h\mathbb P\left( T_k(\lambda)\right)
\end{equation}
where $T_k(\lambda)$ is the event that $(E_k\oplus\cdots\oplus E_1)^\perp$ has type $A_\lambda$. In particular, $B_k$ and $T_k(\lambda)$ are independent and $\mathbb P(B_k)=\frac{k+1}h$.
\end{lemma}

\begin{proof}
By Theorem \ref{thm: Nh(lambda)}, the relative probability of $B_k$ given $T_k(\lambda)$ is
\[
\mathbb P(B_k|T_k(\lambda))=\frac{|\cN_h^p(\lambda)|}{ |\cN_h(\lambda)|}
=\frac{k+1}{ h}.\]
The lemma follows.
\end{proof}


\section{Restricting to exceptional sequences of length $k$} 
The main theorem follows from the following lemma.

\begin{lemma}
For $1\le k_1<k_2<\cdots<k_s\le k$ and $\lambda=(\lambda_0,\cdots,\lambda_k)$, the events $B_{k_1},\cdots,B_{k_s},T_k(\lambda)$ are independent.
\end{lemma}

\begin{proof}
Let $k'<k$. Then for all partitions $\lambda' \vdash n-k'$ we have by \eqref{eq: independence of Bk,Tk}:
\[
	\mathbb P(T_k(\lambda)|T_{k'}(\lambda'))=\mathbb P(T_{k}(\lambda)|B_{k'}T_{k'}(\lambda'))
\]
since the distribution of types $\lambda$ depends only on $\lambda'$. Similarly, the distribution of $B_{k}$ for $k>k'$ depends only on $\lambda'$. Taking $k'=k_1$ we have, by Lemma \ref{lem: partition probability}, the following.
\[
	\mathbb P(B_{k_1}\cdots B_{k_s}T_{k}(\lambda))=\sum_{\lambda' \vdash \,n-k_1} \mathbb P(B_{k_1}T_{k_1}(\lambda'))\mathbb P(B_{k_2}\cdots B_{k_s}T_{k}(\lambda)|B_{k_1}T_{k_1}(\lambda'))
\]
\[
	= \sum_{\lambda' \vdash \,n-k_1} \mathbb P(B_{k_1})\mathbb P(T_{k_1}(\lambda'))\mathbb P(B_{k_2}\cdots B_{k_s}T_{k}(\lambda)|T_{k_1}(\lambda'))
\]
\[
	=\mathbb P(B_{k_1})\mathbb P(B_{k_2}\cdots B_{k_s}T_{k}(\lambda))
\]
By induction on $s$, these events are all independent.
\end{proof}

\begin{theorem}[Theorem \ref{thm A}]\label{thm: thm A}
For $j\le k\le n$ let $B_{k,j}$ be the event that $E_{j}$ is relatively projective in a random exceptional sequence $(E_k,\cdots,E_1)$ of length $k$. Then $\mathbb P(B_{k,j})=\frac{j+1}{n+1}$. Furthermore, these events (for fixed $k$) are independent.
\end{theorem}

\begin{proof}
The number of exceptional sequences of length $k$ is $\sum |\cN_h(\lambda)|$ where the sum is over all partitions $\lambda=(\lambda_0,\cdots,\lambda_k)$ of $n-k$ into $k+1$ parts. For each such $\lambda$, the proportion of those elements of $\cN_h(\lambda)$ in which $E_{j}$ is relatively projective is $\frac{j+1}h$ and these events are independent by the lemma. Since this is a partition of the sample space of all exceptional sequences of length $k$, we can sum over all $\lambda$ to see that the events $B_{k,j}$ are independent with probability $\frac{j+1}h$.
\end{proof}

The independence of probabilities allows us to use the following elementary argument whose proof we include since it is a key step in the argument.

\begin{lemma}\label{lem: probability gives number of clusters} For a random exceptional sequence $(E_k,\cdots,E_1)$ of length $k$, let $C_j$ be the event that the $j$-th term $E_j$ is relatively projective. If the events $C_j$ are independent with probability $P_j$ then the generating function $f(z)=\sum b_j z^j$ for $b_j=$ the number of exceptional sequences of length $k$ with $j$ relatively projective terms is
\[
    f_k(z)=e_k\prod_{j=1}^k (zP_j+1-P_j).
\]
where $e_k$ is the number of exceptional sequences of length $k$.
In particular, the number of signed exceptional sequences of length $k$ is equal to
\[
	f_k(2)=e_k\prod_{j=1}^\ell (1+P_j).
\]
\end{lemma}

\begin{proof} It is easier to explain the multivariable generating function:
\[
    f(z_k,\cdots,z_1)=\sum_\alpha b_\alpha z^\alpha
\]
where the sum is over all multi-indices $\alpha=(a_k,\cdots,a_1)\in \{0,1\}^n$, $z^\alpha=\prod z_i^{a_i}$ and $b_\alpha$ is the number of exceptional sequences of length $k$ in which $E_j$ is relatively projective for $a_j=1$ and not relatively projective for $a_j=0$. Since the proportion of the $e_k$ exceptional sequences of length $k$ in which $E_1$ is relatively projective is $P_1$ and the proportion in which $E_1$ is not relatively projective is $1-P_1$, we get $f(1,1,\cdots,1,z_1)=e_k(P_1z_1+1-P_1)$. Since $E_2$ being relatively projective is independent of this, $f(1,\cdots,1,z_2,z_1)=e_k(P_1z_1+1-P_1)(P_2z_2+1-P_2)$. Continuing in this way, we obtain:
\begin{equation}\label{eq: multivariable f(z) for independent probabilities}
        f(z_k,\cdots,z_1)=e_k\prod_{j=1}^k (P_jz_j+1-P_j).
\end{equation}
Putting all $z_j=z$ we obtain the formula for $f_k(z)$. Every exceptional sequence with $j$ relatively projective terms gives $2^j$ signed exceptional sequences since each relatively projective term can have two signs independently. So, the number of signed exceptional sequences is $f_k(2)$.
\end{proof}

\begin{corollary}\label{cor: generating function for exc seq}
The generating function for $b_{k,p}$, the number of exceptional sequences of length $k$ with $p$ relatively projective terms is
\[
	f_{A_n,k}(z)=\sum b_{k,p}z^p= \sum_{\lambda\vdash\,n-k}N_h(\lambda)\prod_{j=1}^k\left(\frac{(j+1)z+n-j}h
	\right)
\]
\[
	=\frac1{n+1}\binom{n+1}{k+1}\prod_{j=1}^k ((j+1)z+n-j).
\]
\end{corollary}

Recall that a \emph{signed exceptional sequence} \cite{IT13} is an exceptional sequence in which relatively projective terms are allowed to be labeled as positive or negative. Thus, the number of signed exceptional sequences of length $k$ is given by inserting $z=2$ in the above expression.

\begin{corollary}
The number of signed exceptional sequences of length $k$ for $A_n$ with any orientation is equal to
\[
	f_{A_n,k}(2)=\frac1{n+1}\binom{n+1}{k+1}\frac{(n+k+2)!}{(n+2)!}
\]
\end{corollary}

This can also be deduced from known results since signed exceptional sequences are in bijection with ordered partial cluster tilting sets by \cite[Theorem 2.3]{IT13} and the number of such sets for $A_n$ is
\[
	\frac{1}{n+1}\binom{n+1}{k+1}\binom{n+k+2}{n+2}.
\]

Dualizing and reversing the order of an exceptional sequence $(E_k,\cdots,E_1)$ gives an exceptional sequence $(DE_1,\cdots,DE_k)$ for the opposite algebra. An object $E_j$ in the first sequence will be relatively projective if and only if $DE_j$ is relatively injective in the second sequence. This gives the following.

\begin{corollary}\label{cor: injective probability}
Given a random exceptional sequence $(E_1,\cdots,E_k)$ of length $k$ for $A_n$ with any orientation, the probability that $E_j$ is relatively injective is equal to $\frac{j+1}{n+1}$. Furthermore, these events are independent.
\end{corollary}


\section{Rooted labeled forests}

A \emph{rooted labeled forest} of size $n$ is a rooted forest (disjoint union of rooted trees) with $n$ nodes numbered $1$ through $n$. For example, Figure \ref{Fig: example of rooted labeled forest} shows a rooted labeled forest of size $4$.

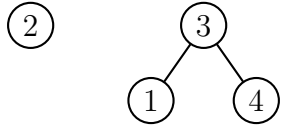
\begin{figure}[htbp]
\begin{center}
\begin{tikzpicture}
\coordinate (A) at (-1,1);
\coordinate (B) at (.6,0);
\coordinate (C) at (1.3,1); 
\coordinate (D) at (2,0); 
\draw[thick] (B)--(C)--(D);
\foreach \x in {B,C,D}\draw[white,fill] (\x) circle[radius=3mm];
\draw[thick] (A) node{$2$} circle[radius=3mm];
\draw[thick] (B) node{$1$} circle[radius=3mm];
\draw[thick] (C) node{$3$} circle[radius=3mm];
\draw[thick] (D) node{$4$} circle[radius=3mm];
\end{tikzpicture}
\caption{This is a rooted labeled forest of size $4$ with roots at the top. Vertices $v_2,v_3$ are the roots. $v_1$ is a descending vertex since $1<3$. $v_4$ is an ascending vertex since $4>3$.}
\label{Fig: example of rooted labeled forest}
\end{center}
\end{figure}

In \cite{IS} we showed the following theorem which was suggested to us by Olivier Bernardi. From now on, we number our exceptional sequences from left to right to agree with this and other references. By ``$A_n$ with linear orientation'' we mean the following quiver.
\[
	A_n:\quad 1\leftarrow 2\leftarrow 3\leftarrow \cdots\leftarrow n.
\]
The two linear orientations of $A_n$ are isomorphic and the other one is used in \cite{IS}.

\begin{theorem}\cite{IS}\label{IS theorem}
There is a bijection between rooted labeled forests and complete exceptional sequences for the quiver of type $A_n$ with linear orientation. Furthermore, if $F$ is the forest corresponding to $E_\ast=(E_1,\cdots,E_n)$ then:
\begin{enumerate}
\item[a)] $E_k$ is relatively projective in $E_\ast$ if and only if the corresponding vertex $v_k$ in $F$ is either a root of $F$ or a descending vertex (a child of $v_\ell$ where $\ell>k$)
\item[b)] $E_k$ is relatively injective if and only if $v_k$ is a root or an ascending vertex of $F$.
\end{enumerate}
\end{theorem}

Combining this with our main result and the fact that every vertex of a rooted labeled forest is either descending, ascending, or a root, we get the following.

\begin{corollary}\label{cor: probabilities for forests}
For a random rooted labeled forest of size $n$ and $1\le k\le n$, let $D_k$ be the event that $v_k$ is descending. Then these events are independent and have probability
\[
	\mathbb P(D_k)=\frac{n-k}{n+1}.
\] 
\end{corollary}

This gives a probabilistic interpretation of the one variable generating function for the number of descending vertices and, by duality, ascending vertices. Thus if $c_p$ is the number of rooted labeled forests with $p$ ascending vertices then
\[
	f(a)=\sum c_p a^p=(n+1)^{n-1}\prod_{k=1}^n \left(
	\frac{k+1+(n-k)a}{n+1}
	\right)
\]
\[
	=\prod_{k=1}^{n-1}(k+1+(n-k)a)
\]


\section{Other Dynkin quivers}

Our main theorem, which holds for quivers of type $A_n$ with any orientation does not seem to hold for other simply-laced Dynkin quivers of rank 3 or more. We will compute the probability distribution for $D_4$ using a simple recursive formula. It is based on the following theorem which follows easily from \cite{Ringel}. For this formula, we need to write a complete exceptional sequence in the standard way: $(E_1,\cdots,E_n)$. For any not necessarily connected Dynkin quiver $\Delta$ we define the $n$ variable generating function $f_\Delta(z_\ast)$, as in the proof of Lemma \ref{lem: probability gives number of clusters}, by
\[
	f_\Delta(z_1,\cdots,z_n)=\sum_\alpha c_\alpha z^\alpha
\]
where the sum is over all multi-indices $\alpha=(a_1,\cdots,a_n)$ where each $a_i\in \{0,1\}$, $z^\alpha=z^{a_1}z^{a_2}\cdots z^{a_n}$ and $c_\alpha$ is the number of complete exceptional sequences $(E_1,\cdots,E_n)$ in which $E_i$ is relatively projective for $a_i=1$ and not relatively projective for $a_i=0$.

We have the following theorem which is essentially due to \cite{Ringel}.

\begin{theorem}\label{thm: recursive formula for fAn}
\[
f_\Delta(z_1,\cdots,z_n)=\frac12(2z_n+ h-2) \sum_{i=1}^n f_{\Delta_i}(z_1,\cdots,z_{n-1})
\]
where $h$ is the Coxeter number of $\Delta$ and $\Delta_i$ denotes $\Delta$ with vertex $i$ deleted.
\end{theorem}

We note that, as in Equation \eqref{eq: multivariable f(z) for independent probabilities}, the statement that the events $B_k$ ($E_k$ being relatively projective) are independent if and only if the $n$ variable generating function is a product of $n$ linear terms in the variables $z_i$ separately. For example, take $A_3$. Then $h=4$. So,
\[
	f_{A_3}(z_1,z_2,z_3)=(z_3+1)\left[
	f_{A_1\times A_1}(z_1,z_2)+2f_{A_2}(z_1,z_2)
	\right]
\]
\[
	=(z_3+1)[2z_1z_2+2(2z_2+1)z_1]= 2z_1[3z_2z_3+3z_2+z_3+1]
\]	
\[=2z_1(3z_2+1)(z_3+1).
\]
Since this is a product of linear factors in the $z_i$, this verifies Theorem \ref{thm A} in this case. 

More generally, Theorem \ref{thm A} and Equation \eqref{eq: multivariable f(z) for independent probabilities} imply the following
\[
    f_{A_n}(z_1,\cdots,z_n)=z_1\prod_{i=1}^{n-1}(i+(n+1-i)z_{i+1}).
\]

In \cite[Corollary 5.10]{IS2} we show for type $B_n$ and $C_n$, the analogue of Theorem \ref{thm A} holds and we obtain the following.
\[
    f_{B_n}(z_1,\cdots,z_n)=\prod_{i=1}^{n}(i-1+(n+1-i)z_{i}).
\]
Also, it is easy to see that the analogue of Theorem \ref{thm A} holds for type $G_2$. We believe that this statement fails for every other type ($D_n, E_6,E_7,E_8,F_4$). 

For $D_4$, we have the following.
\[
	f_{D_4}(z_1,z_2,z_3,z_4)=(z_4+2)[f_{A_1\times A_1\times A_1}(z_1,z_2,z_3)+3f_{A_3}(z_1,z_2,z_3)
\]
\[
	=(z_4+2)[
	6z_1z_2z_3+6z_1(3z_2z_3+3z_2+z_3+1)
	]
\]
\[
	=6z_1(z_4+2)(4z_2z_3+3z_2+z_3+1)
\]
which does not factor any further. Thus, Theorem \ref{thm A} does not hold for $D_4$. The events $B_2,B_3$ are not independent.

However, when the variables are collapsed to one variable $z$, we obtain
\[
	f_{D_4}(z)=6z(z+2)(2z+1)^2
\]
as expected (See \cite{FR}.) This is the polynomial which gives $f_{D_4}(1)=162$, the number of complete exceptional sequences and $f_{D_4}(2)=1200$, the number of complete signed exceptional sequences which is $4!$ times the number of clusters \cite{IT13}.


\section{Further research} For any Dynkin quiver, the number of exceptional sequences is given by a product formula similar to the one for $A_n$. This suggests that there may be independent events being counted. It would be very interesting to find what these events might be in other Dynkin types. This is a challenging problem in types $D_4, D_5$ and we have not done it in other types.

For rooted labeled forests that correspond to complete exceptional sequences of type $A_n$ we were not able to find a definition of a partial rooted labeled forest corresponding to exceptional sequences of length $k<n$.
\bigskip

The author thanks Steve Hermes and Keith Merrill for many conversations about probability distributions of representations. He also thanks Olivier Bernardi and Emre Sen for very productive discussions about the corresponding distribution of rooted labeled forests, leading to \cite{IS} and \cite{IS2}. The author also thanks Gordana Todorov for many discussions about signed exceptional sequences before and after \cite{IT13}. The author is supported by the Simons Foundation.

\end{document}